\documentclass[12pt,leqno]{amsart}
\usepackage{amsmath}
\usepackage{amsthm}
\usepackage{amssymb}
\theoremstyle{plain} %text of this environment is typesetted in italics
\newtheorem{theorem}{\indent\sc Theorem}[section]
\newtheorem{lemma}[theorem]{\indent\sc Lemma}

\newtheorem{proposition}[theorem]{\indent\sc Proposition}

\theoremstyle{definition} %text of this environment is typesetted in roman letters


\newcommand{\C}{{\mathbf{C}}}

\renewcommand{\P}{{\mathbf{P}}}

\renewcommand{\log}{\mathrm{log}}

\newcommand{\Z}{\mathbf{Z}}
\numberwithin{equation}{section}

\begin{document}

\title[Two meromorphic mappings sharing $2n+2$ hyperplanes]{Two meromorphic mappings sharing $2n+2$ hyperplanes regardless of multiplicity} 

\maketitle

\begin{center}
Si Duc Quang$^a$ and Le Ngoc Quynh$^b$
\end{center}

{\small
\textit{\begin{center}
$^a$ Department of Mathematics, Hanoi National University of Education,\\
136 Xuan Thuy street, Cau Giay, Hanoi, Vietnam\\
email address: quangsd@hnue.edu.vn
\end{center}
\begin{center}
$^b$ Faculty of Education, An Giang University, 18 Ung Van Khiem,\\
Dong Xuyen, Long Xuyen, An Giang, Vietnam\\
email address: nquynh1511@gmail.com
\end{center}}}


\begin{abstract} 
Nevanlinna showed that two non-constant meromorphic functions on $\C$ must be linked by a M\"{o}bius transformation if they have the same inverse images counted with multiplicities for four distinct values. After that this results is generalized by Gundersen to the case where two meromorphic functions share two values ignoring multiplicity and share other two values with multiplicities trucated by $2$. Previously, the first author proved that for $n\ge 2,$ there are at most two linearly nondegenerate meromorphic mappings of $\C^m$ into $\P^n(\C)$ sharing $2n+2$ hyperplanes ingeneral position ignoring multiplicity. In this article, we will show that if two meromorphic mappings $f$ and $g$ of $\C^m$ into $\P^n(\C)$ share $2n+1$ hyperplanes ignoring multiplicity and another hyperplane with multiplicities trucated by $n+1$ then the map $f\times g$ is algebraically degenerate.
\end{abstract}

\section*{Introduction}
In $1926$, R. Nevanlinna \cite{N} showed that if two distinct nonconstant meromorphic functions $f$ and $g$ on the complex plane $\C$  have the same inverse images for four distinct values then $g$ is a special type of linear fractional transformation of $f$.

The above result is usually called the four values theorem of Nevanlinna. In 1983, Gundersen \cite{G} improved the result of Nevanlinna by proving the following.

\vskip0.2cm
{\sc Theorem A} (Gundersen \cite{G}).\ {\it Let $f$ and $g$ be two distinct non-constant meromorphic functions and let $a_1,a_2,a_3,a_4$ be four distinct values in $\C\cup\{\infty\}$. Assume that 
$$\min\{\nu^0_{f-a_i}, 1\} = \min\{\nu^0_{g-a_i},1\}\text{ for }i = 1, 2\text{ and }\nu^0_{f-a_j}= \nu^0_{g-a_ j} \text{ and } j = 3, 4$$
$\bigl ($outside a discrete set of counting function regardless of multiplicity is equal to $o(T(r,f ))\bigl )$. Then $\nu^0_{f-a_i}=\nu^0_{g-a_i}$ for all $i\in\{1,\dots , 4\}$.}

\vskip0.2cm
In this article, we will extend and improve the above results of Nevanlinna and Gundersen to the case of meromorphic mappings into  $\P^n(\C)$. To state our results, we firstly give some following.

Take two meromorphic mapping $f$ and $g$ of $\C^m$ into $\P^n(\C)$. Let $H$ be a hyperplanes of $\P^n(\C)$ such that $(f,H)\not\equiv 0$ and $(g,H)\not\equiv 0$. Let $d$ be an positive integer or $+\infty$. We say that $f$ and $g$ share the hyperplane $H$ with multiplicity truncated by $d$ if the following two conditions are satisfied:
$$\min\ (\nu_ {(f,H)},d)=\min\ (\nu_{(g,H)},d) \text{ and }f(z) = g(z)\text{ on }f^{-1}(H).$$
If $d=1$, we will say that $f$ and $g$ share $H$ ignoring multiplicity. If $d=+\infty$, we will say that $f$ and $g$ share $H$ with counting multiplicity.

Recently, Chen - Yan \cite{CY1} and S. D. Quang \cite{SQ1} showed that two meromorphic mappings of $\C^m$ into $\P^n(\C)$ must be identical if they share $2n+3$ hyperplanes in general position ignoring multiplicity. In 2011, Chen - Yan considered the case of meromorphic mappings sharing only $2n+2$ hyperplanes, and they showed that

{\sc Theorem B} (see \cite[Main Theorem]{CY2}). \textit{Let $f,g$ and $h$ be three linearly nondegenerate meromorphic mappings of $\C^m$ into $\P^n(\C )$. Let $H_1,...,H_{2n+2}$ be $2n+2$ hyperplanes of $\P^n(\C )$ in general position with 
$$\dim  f^{-1}(H_i \cap H_j) \leqslant m-2 \quad (1 \leqslant i<j \leqslant 2n+2).$$ 
Assume that $f,g$ and $h$ share $H_1,...,H_{2n+2}$ with multiplicity truncated by level $2$. Then the map $f\times g\times h$ is linearly degenerate.}
 
Independently, in 2012 S. D. Quang \cite{SQ2} proved a finiteness theorem for meromorphic mappings sharing $2n+2$ hyperplanes without multiplicity as follows.

\vskip0.2cm
{\sc Theorem C} ( see \cite[Theorem 1.1]{SQ2}). \textit{Let $f,g$ and $h$ be three meromorphic mappings of $\C^m$ into $\P^n(\C )$. Let $H_1,...,H_{2n+2}$ be $2n+2$ hyperplanes of $\P^n(\C )$ in general position with 
$$\dim  f^{-1}(H_i \cap H_j) \leqslant m-2 \quad (1 \leqslant i<j \leqslant 2n+2).$$ 
Assume that $f,g$ and $h$ share $H_1,...,H_{2n+2}$ ignoring multiplicity. If $f$ is linearly nondegenerate and $n\ge 2$ then
$$ f=g\text{ or }g=h\text{ or }h=f. $$}
\indent
The above theorem means that there are at most two linearly nondegenerate meromorphic mappings of $\C^m$ into $\P^n(\C )$ sharing $2n+2$ hyperplanes in general position regardless of multiplicity. In this paper, we will show that there is an algebraic relation among them if they share at least one of these hyperplanes with multiplicity truncated by level $n+1$. Namely, we will prove the following.

\vskip0.2cm
{\sc Main Theorem.} \textit{Let $f$ and $g$ be two meromorphic mappings of $\C^m$ into $\P^n(\C )$. Let $H_1,...,H_{2n+2}$ be $2n+2$ hyperplanes of $\P^n(\C )$ in general position with 
$$\dim  f^{-1}(H_i \cap H_j) \leqslant m-2 \quad (1 \leqslant i<j \leqslant 2n+2).$$ 
Assume that $f$ and $g$ share $H_1,...,H_{2n+1}$ ignoring multiplicity and share $H_{2n+2}$ with multiplicity truncated by $n+1$. Then the map $f\times g:\C^m\rightarrow\P^n(\C)\times\P^n(\C)$ is algebraically degenerate.}
\vskip0.2cm

In the last section of this paper, we will consider the case of two meromorphic mappings sharing two different families of hyperplanes. We will also give an algebraically degeneracy theorem for that case.

{\sc Acknowledgements.} This work was done during a stay of the first author at Vietnam Institute for Advanced Study in Mathematics. He would like to thank the institute for support. This work is also supported in part by a NAFOSTED grant of Vietnam.

\section{Basic notions and auxiliary results from Nevanlinna theory}

\noindent
{\bf 2.1.}\ We set $||z|| = \big(|z_1|^2 + \dots + |z_m|^2\big)^{1/2}$ for
$z = (z_1,\dots,z_m) \in \C^m$ and define
\begin{align*}
B(r) := \{ z \in \C^m : ||z|| < r\},\quad
S(r) := \{ z \in \C^m : ||z|| = r\}\ (0<r<\infty).
\end{align*}

Define 
\begin{align*}
\sigma(z) :=& \big(dd^c ||z||^2\big)^{m-1}\quad \quad \text{and}\\
\eta (z):=& d^c \text{log}||z||^2 \land \big(dd^c \text{log}||z||^2\big)^{m-1}
\text{on} \quad \C^m \setminus \{0\}.
\end{align*}
{\bf 2.2.}\ Let $F$ be a nonzero holomorphic function on a domain $\Omega$ in $\C^m$. For a set $\alpha = (\alpha_1,...,\alpha_m) $ 
of nonnegative integers, we set $|\alpha|=\alpha_1+...+\alpha_m$ and 
$\mathcal {D}^\alpha F=\dfrac {\partial ^{|\alpha|} F}{\partial ^{\alpha_1}z_1...\partial ^{\alpha_m}z_m}.$
We define the map $\nu_F : \Omega \to \Z$ by
$$\nu_F(z):=\max\ \{l: \mathcal {D}^\alpha F(z)=0 \text { for all } \alpha \text { with }|\alpha|<l\}\ (z\in \Omega).$$

We mean by a divisor on a domain $\Omega$ in $\C^m$ a map $\nu : \Omega \to \Z$ such that, for each $a\in \Omega$, there are nonzero 
holomorphic functions $F$ and $G$ on a connected neighborhood $U\subset \Omega$ of $a$ such that $\nu (z)= \nu_F(z)-\nu_G(z)$ 
for each $z\in U$ outside an analytic set of dimension $\leqslant m-2$. Two divisors are regarded as the same if they are identical 
outside an analytic set of dimension $\leqslant m-2$. For a divisor $\nu$ on $\Omega$ we set $|\nu| := \overline {\{z:\nu(z)\ne 0\}},$ which
is a purely $(m-1)$-dimensional analytic subset of $\Omega$ or empty set.

Take a nonzero meromorphic function $\varphi$ on a domain $\Omega$ in $\C^m$. For each $a\in \Omega$, we choose nonzero 
holomorphic functions $F$ and $G$ on a neighborhood $U\subset \Omega$ such that $\varphi = \dfrac {F}{G}$ on $U$ and 
$dim (F^{-1}(0)\cap G^{-1}(0))\leqslant m-2,$ and we define the divisors $\nu_\varphi,\ \nu^\infty_\varphi$ by 
$ \nu_\varphi := \nu_F,\  \nu^\infty_\varphi :=\nu_G$, which are 
independent of choices of $F$ and $G$ and so globally well-defined on $\Omega$.

\noindent
{\bf 2.3.}\ For a divisor $\nu$ on $\C^m$ and for a positive integer $M$ or $M= \infty$, we define the counting function of $\nu$ by

\begin{align*}
\nu^{(M)}(z)&=\min\ \{M,\nu(z)\},\\
n(t)& =
\begin{cases}
\int\limits_{B(t)}
\nu(z) \sigma & \text  { if } m \geq 2,\\
\sum\limits_{|z|\leq t} \nu (z) & \text { if }  m=1. 
\end{cases}\\
N(r,\nu)& =\int\limits_1^r \dfrac {n(t)}{t^{2m-1}}dt \quad (1<r<\infty).
\end{align*}

For a meromorphic function $\varphi$ on $\C^m$, we set $N_\varphi (r)=N(r,\nu_\varphi )$ and $N_\varphi^{[M]} (r)=N(r,\nu_\varphi^{[M]}).$ We will omit the character $^{[M]}$ if $M=\infty$.

\noindent
{\bf 2.4.}\ Let $f : \C^m \longrightarrow \P^n(\C)$ be a meromorphic mapping.
For arbitrarily fixed homogeneous coordinates
$(w_0 : \dots : w_n)$ on $\P^n(\C)$, we take a reduced representation
$f = (f_0 : \dots : f_n)$, which means that each $f_i$ is a  
holomorphic function on $\C^m$ and 
$f(z) = \big(f_0(z) : \dots : f_n(z)\big)$ outside the analytic set
$I(f)=\{ f_0 = \dots = f_n= 0\}$ of codimension $\geq 2$.
Set $\Vert f \Vert = \big(|f_0|^2 + \dots + |f_n|^2\big)^{1/2}$.

The characteristic function of $f$ is defined by 
\begin{align*}
T_f(r) = \int\limits_{S(r)} \text{log}\Vert f \Vert \eta -
\int\limits_{S(1)}\text{log}\Vert f\Vert \eta.
\end{align*}

Let $H$ be a hyperplane in $\P^n(\C)$ given by $H=\{a_0\omega_0+...+a_n\omega_n=0\},$ where $a:=(a_0,...,a_n)\ne (0,...,0)$. 
We set  $(f,H)=\sum_{i=0}^na_if_i$. It is easy to see that the divisor $\nu_{(f,H)}$ does not depend on the choices of reduced representation of $f$ and coefficients   $a_0,..,a_n$. Moreover, we define the proximity function of $f$ with respect to $H$ by 
$$m_{f,H}(r)=\int_{S(r)}\log \dfrac {||f||\cdot ||H||}{|(f,H)|}\eta-\int_{S(1)}\log \dfrac {||f||\cdot ||H||}{|(f,H)|}\eta,$$
where $||H||=(\sum_{i=0}^n|a_i|^2)^{\frac {1}{2}}.$

\noindent
{\bf 2.5.}\ Let $\varphi$ be a nonzero meromorphic function on $\C^m$, which is occasionally regarded as a meromorphic map into $\P^1(\C)$. The proximity function of $\varphi$ is defined by
$$m(r,\varphi):=\int_{S(r)}\log ^+|\varphi|\eta,$$
where $\log^+t=\max\{0,\log t\}$ for $t>0$.
The Nevanlinna characteristic function of $\varphi$ is defined by
$$T(r,\varphi )=N_{\frac{1}{\varphi}}(r)+m(r,\varphi ).$$
There is a fact that
$$T_\varphi (r)=T(r,\varphi )+O(1).$$
The meromorphic function $\varphi$ is said to be small with respect to $f$ iff
$||\ T(r,\varphi )=o(T_f(r))$. 

Here as usual, by the notation $``|| \ P"$  we mean the assertion $P$ holds for all $r \in [0,\infty)$ excluding a Borel subset $E$ of the interval $[0,\infty)$ with $\int_E dr<\infty$.

The following plays essential roles in Nevanlinna theory (see \cite{NO}).
\begin{theorem}[First main theorem]\label{1.1} Let $f: \C^m \to \P^n(\C)$ be a meromorphic mapping and let $H$ be a hyperplane in $\P^n(\C)$ such that $f(\C^m)\not\subset H$. Then 
$$
N_{(f,H)}(r)+m_{f,H}(r)=T_f(r)\ (r>1).
$$
\end{theorem}

\begin{theorem}[Second main theorem]\label{1.2} Let $f: \C^m \to \P^n(\C)$ be a linearly nondegenerate meromorphic mapping and $H_1,...,H_q$ 
be hyperplanes of $\P^n(\C)$ in general position. Then 
$$
||\ \ (q-n-1)T_f(r) \leqslant \sum_{i=1}^q N^{[n]}_{(f,H_i)}(r)+o(T_f(r)).
$$
\end{theorem}

\begin{lemma}[Lemma on logarithmic derivative]\label{1.3}
Let $f$ be a nonzero meromorphic function on $\C^m.$ Then 
$$
\biggl|\biggl|\quad m\biggl(r,\dfrac{\mathcal{D}^\alpha (f)}{f}\biggl)=O(\log^+T_f(r))\ (\alpha\in \Z^m_+).
$$
\end{lemma}

\noindent
{\bf 2.6.}\ Let $h_1, h_2,..., h_p$ be finitely many nonzero meromorphic functions on $\C^m$. By a rational function in logarithmic derivatives of $h_j'$s we
mean a nonzero meromorphic function $\varphi$ on $\C^m$ which is represented as
$$ \varphi =\dfrac{P(\cdots, \frac{\mathcal D^\alpha h_j}{h_j},\cdots )}{Q(\cdots, \frac{\mathcal D^\alpha h_j}{h_j},\cdots )} $$
with polynomials $P(\cdots, X^\alpha,\cdots)$ and $Q(\cdots, X^\alpha,\cdots)$
\begin{proposition}[{see \cite[Proposition 3.4]{Fu99}}]\label{1.4}
Let $h_1, h_2,..., h_p\ (p\ge 2)$ be nonzero meromorphic functions on $\C^m$. Assume that
$$ h_1+h_2+\cdots +h_p=0 $$
Then, the set $\{1,...,p\}$ of indices has a partition
$$ \{1,...,p\}=J_1\cup J_2\cup\cdots\cup J_k,\sharp J_\alpha \ge 2\ \forall\ \alpha, J_\alpha\cap J_\beta =\emptyset\text{ for }\alpha\ne\beta $$
such that, for each $\alpha$,
\begin{align*}
\mathrm{(i)}&\ \sum_{i\in J_\alpha}h_i=0,\\ 
\mathrm{(ii)}&\ \ \text{$\frac{h_i'}{h_i}\ (i,i'\in J_\alpha)$ are rational functions in logarithmic derivatives of $h_j'$s} 
\end{align*}
\end{proposition} 

\section{Algebraic degeneracy of two meromorphic mappings}

In order to prove the main theorem, we need the following algebraic propositions.

Let $H_1,...,H_{2n+1}$ be $(2n+1)$ hyperplanes of $\P^n(\C)$ in general position given by
$$ H_i:\ \ x_{i0}\omega_0+x_{i1}\omega_1+\cdots +x_{in}\omega_n=0\ (1\le i\le 2n+1). $$
We consider the rational map $\Phi :\P^n(\C)\times \P^n(\C)\longrightarrow \P^{2n}(\C) $ as follows:

For $v=(v_0:v_1\cdots :v_n),\ w=(w_0:w_1:\cdots :w_n)\in \P^n(\C)$, we define the value $\Phi (v,w)= (u_0:\cdots :u_{2n+1})\in\P^{2n}(\C)$ by
$$ u_i=\frac{x_{i0}v_0+x_{i1}v_1+\cdots +x_{in}v_n}{x_{i0}w_0+x_{i1}w_1+\cdots +x_{in}w_n}. $$
\begin{proposition}[{see \cite[Proposition 5.9]{Fu99}}]\label{2.1}
The map $\Phi$ is a birational map of $\P^n(\C)\times \P^n(\C)$ onto $\P^{2n}(\C)$.
\end{proposition}

Let $f$ and $g$ be two meromorphic mappings of $\C^m$ into $\P^n(\C)$ with reduced representations
$$ f=(f_0:\cdots :f_n)\ \text{ and }\ g=(g_0:\cdots :g_n). $$
Define $h_i=\frac{(f,H_i)/f_0}{(g,H_i)/g_0}\ (1\le i\le 2n+1)$ and $h_I=\prod_{i\in I}h_i$ for each subset $I$ of $\{1,...,2n+1\}.$
Set $\mathcal I=\{I=(i_1,...,i_n)\ ;\ 1\le i_1<\cdots <i_n\le 2n+1\}$. We have the following proposition
\begin{proposition}\label{2.2}
If there exist constants $A_I$, not all zero, such that
$$ \sum_{I\in\mathcal I}A_Ih_I\equiv 0 $$
then the map $f\times g$ into $\P^n(\C)\times\P^n(\C)$ is algebraically degenerate.
\end{proposition}
\begin{proof}
For $v=(v_0:v_1\cdots :v_n),\ w=(w_0:w_1:\cdots :w_n)\in \P^n(\C)$, we define the map $\Phi (v,w)= (u_0:\cdots :u_{2n+1})\in\P^{2n}(\C)$ as above.  By Proposition \ref{3.1}, $\Phi$ is birational function. This implies that the function 
$$\sum_{I\in\mathcal I}A_I\frac{x_{i0}v_0+x_{i1}v_1+\cdots +x_{in}v_n}{x_{i0}w_0+x_{i1}w_1+\cdots +x_{in}w_n}$$
is a nonzero rational function. It follows that 
$$ Q(v_0,...,v_n,w_0,...,w_n)=\sum_{I\in\mathcal I}A_I\left (\prod_{i\in I}\sum_{j=0}^{n}x_{ij}v_j\right )\times \left (\prod_{i\in I^c}\sum_{j=0}^{n}x_{ij}w_j\right ),$$
where $I^c=\{1,....,2n+1\}\setminus I$, is a nonzero polynomial. Since the assumption of the proposition, it is clear that
$$ Q(f_0,...,f_n,g_0,...,g_n)\equiv 0. $$
Hence $f\times g$ is algebraically degenerate.
\end{proof}

\begin{proposition}\label{2.3}
Let $f,g$ be two meromorphic mappings of $\C^m$ into $\P^n(\C)$. Let $\{H_i\}_{i=1}^{2n+2}$ be $(2n+2)$ hyperplanes of $\P^n(\C)$ in general position as in Main Theorem. Suppose that the map $f\times g$ is algebraically nondegenerate. Then the following assertions hold:

$\mathrm{(a)}$ $||\ T_f(r)=O(T_g(r))$ and $||\ T_g(r)=O(T_f(r))$.

$\mathrm{(b)}$ $m\biggl (r,\dfrac{(f,H_i)}{(g,H_i)}\dfrac{(g,H_j)}{(f,H_j)}\biggl )=o(T_f(r))\ \forall 1\le i,j\le 2n+2$.
\end{proposition}
\begin{proof}
(a). By the supposition the map $f\times g$ is algebraically non-degenerate, both $f$ and $g$ are linearly nondegenerate. Assume that $f,g$ have reduced representations
$$ f=(f_0:\cdots :f_n),\ g=(g_0:\cdots :g_n),$$
and the hyperplane $H_i\ (1\le i\le 2n+2)$ is given by
$$ H_i=\{(w_0:\cdots :w_n)\ ;\ a_{i0}w_0+\cdots +a_{in}w_n=0\}. $$
By Theorem \ref{1.2} we have
\begin{align*}
\bigl |\bigl |\ (n+1)T_f(r)&\le\sum_{i=1}^{2n+2}N_{(f,H_i)}^{[n]}(r)+o(T_f(r))\\
&\le n\cdot\sum_{i=1}^{2n+2}N_{(g,H_i)}^{[1]}(r)+o(T_f(r))\\
&\le n(2n+2)(T_g(r))+o(T_f(r)).
\end{align*}
Then we have $||\ T_f(r)=O(T_g(r))$. Similarly we also have $||\ T_g(r)=O(T_f(r))$.  We have the first assertion of the proposition.

(b). Since $\sum_{k=0}^n a_{ik}f_k-\dfrac{f_0h_i}{g_0}\cdot \sum_{k=0}^n a_{ik}g_k=0\ (1\le i \le 2n+2),$ it implies that 
\begin{align}\label{2.4}
\Phi :=\det \ (a_{i0},...,a_{in},a_{i0}h_i,...,a_{in}h_i; 1\le i \le 2n+2)\equiv 0.
\end{align}

For each subset $I\subset \{1,2,...,2n+2\},$ put $h_I=\prod_{i\in I}h_i$. Denote by $\mathcal {I}$ the set 
$$ \mathcal I=\{I=(i_1,...,i_{n+1})\ ;\ 1\le i_1<\cdots <i_{n+1}\le 2n+2\}. $$
For each $I=(i_1,...,i_{n+1})\in \mathcal {I}$, define 
\begin{align*}
A_I=(-1)^{\frac {(n+1)(n+2)}{2}+i_1+...+i_{n+1}}&\times \det (a_{i_rl};1\le r \le n+1,0\le l \le n)\\
&\times\det (a_{j_sl};1\le s \le n+1,0\le l \le n),
\end{align*}
where $J=(j_1,...,j_{n+1})\in \mathcal {I}$ such that $I \cup J =\{1,2,...,2n+2\}.$

We denote by $\mathcal M$ the field of all meromorphic functions on $\C^m$, and denote by $G$ the group of all nonzero functions $\varphi$ so that $\varphi^m$ is a rational function in logarithmic derivatives of ${h_i}'$s for some positive integers $m$. We denote by $\mathcal H$ the subgroup of the group $\mathcal M/G$ generated by elements $[h_1],...,[h_{2n+2}]$.

Hence $\mathcal H$ is a finitely generated torsion-free abelian group. We call $(x_1,...,x_p)$ a basis of $\mathcal H$. Then for each $i\in\{1,...,2n+2\}$, we have
$$ [h_i] =x_1^{t_{i1}}\cdots x_p^{t_{ip}}.$$
Put $t_i=(t_{i1},...,t_{ip})\in \Z^p$ and denote by $``\leqslant "$ the lexicographical order on $\Z^p$. Without loss of generality, we may assume that
$$ t_1 \leqslant  t_2\leqslant \cdots\leqslant t_{2n+2}.$$

Now the equality (\ref{2.4}) implies that 
$$\sum_{I\in \mathcal {I}}A_Ih_I=0.$$
Applying Proposition \ref{1.4} to meromorphic mappings $A_Ih_I\ (I\in\mathcal I)$, then we have a partition $\mathcal I=\mathcal I_1\cup\cdots\cup \mathcal I_k$ with $ \mathcal I_\alpha\ne\emptyset$ and $\mathcal I_\alpha\cap\mathcal I_\beta =\emptyset$ for $\alpha\ne\beta$ such that for each $\alpha$,
\begin{align}
\label{2.5}
&\sum_{I\in\mathcal I_\alpha}A_Ih_I\equiv 0,\\ 
\label{2.6}
& \frac{A_{I'}h_{I'}}{A_Ih_I}\ (I,I'\in\mathcal I_\alpha)\text{ are rational functions in logarithmic derivatives of ${A_{J}h_J}'$s}. 
\end{align}
Moreover, we may assume that $I_\alpha$ is minimal, i.e., there is no proper subset $J_\alpha\subsetneq I_\alpha$ with $\sum_{I\in\mathcal J_\alpha}A_Ih_I\equiv 0$.

We distinguish the following two cases:

\textit{Case 1.} Assume that there exists an index $i_0$ such that $t_{i_0}<t_{i_0+1}$. We may assume that $i_0\le n+1$ (otherwise we consider the relation $``\geqslant "$ and change indices of $\{h_1,...,h_{2n+2}\}$).

Assume that $I=(1,2,...,n+1)\in\mathcal I_1$. By the assertion (\ref{2.6}), for each $J=(j_1,...,j_{n+1})\in \mathcal I_1\ (1\le j_1<\cdots< j_{n+1}\le 2n+2)$, we have $[h_{I}]=[h_{J}]$. This implies that
$$ t_1+\cdots + t_{n+1}=t_{j_1}+\cdots +t_{j_{n+1}}. $$
This yields that $t_{j_i}=t_i\ (1\le i\le n+1)$. 

Suppose that $j_{i_0}>i_0$, then $t_{i_0}<t_{i_0+1}\leqslant t_{j_{i_0}}$. This is a contradiction. Therefore $j_{i_0}=i_0$, and hence $ j_1=1,...,j_{i_0-1}=i_0-1.$ We conclude that $J=(1,...,i_0,j_{i_0+1},...,j_{n+1})$ and $i_0\le n+1$ for each $J\in\mathcal I_1.$

By (\ref{2.6}), we have
$$ \sum_{I\in\mathcal I_1}A_Ih_I=h_{i_0}\sum_{I\in\mathcal I_1}A_Ih_{I\setminus\{i_0\}}\equiv 0. $$
Thus 
$$\sum_{I\in\mathcal I_1}A_Ih_{I\setminus\{i_0\}}\equiv 0.$$
Then Proposition \ref{2.2} shows that $f\times g$ is algebraically degenerate. It contradicts to the supposition.

\textit{Case 2.} Assume that $t_1=\cdots =t_{2n+2}$. It follows that $\frac{h_I}{h_J}\in G$ for any $I,J\in\mathcal I$. Then we easily see that $\frac{h_i}{h_j}\in G$ for all $1\le i,j\le 2n+2.$ Hence, there exists a positive integer $m_{ij}$ such that $\left (\frac{h_i}{h_j}\right )^{m_{ij}}$ is a rational funtion in logarithmic derivatives of ${h_s}'$s. Therefore, by lemma on logarithmic derivatives, we have
\begin{align*}
\biggl |\biggl |\ \ m\bigl (r,\frac{h_i}{h_j}\bigl )&=\frac{1}{m_{ij}}m\bigl (r,\left (\frac{h_i}{h_j}\right )^{m_{ij}}\bigl )+O(1)\\
&=O\biggl (\max m\bigl (r,\frac{\mathcal D^\alpha (h_s)}{h_s}\bigl )\biggl )+O(1)=o(\max T(r,h_s))+O(1)\\
&=o\biggl (\max T\bigl (r,\dfrac{(f,H_s)}{f_0}\bigl )\biggl )+o\biggl (\max T\bigl (r,\dfrac{(g,H_s)}{g_0}\bigl )\biggl )+O(1)\\
&=o(T_f(r))+o(T_g(r))=o(T_f(r)).
\end{align*}
Therefore, we have
$$
  m\biggl (r,\dfrac{(f,H_i)}{(g,H_i)}\dfrac{(g,H_j)}{(f,H_j)}\biggl )=o(T_f(r))\ \forall 1\le i,j\le 2n+2.
$$
The second assertion is proved.
\end{proof}

\begin{proposition}\label{2.7}
Let $f,g:\C^m\rightarrow \P^n(\C)$ be two meromorphic mappings and let $\{H_i\}_{i=1}^{2n+2}$ be $2n+2$ hyperplanes of $\P^n(\C)$ in general position with
$$\dim  f^{-1}(H_i \cap H_j) \leqslant m-2 \quad (1 \leqslant i<j \leqslant 2n+2).$$ 
Assume that $f$ and $g$ share $H_i\ (1\le i\le 2n+2)$ ignoring multiplicity. Suppose that the map $f\times g$ is algebraically nondegenerate. Then for every $i=1,...,2n+2,$ the following assertions hold

$\mathrm{(i)}$ \ $||\ T_f(r)=N_{(f,H_i)}(r)+o(T_f(r))$ and $||\ T_g(r)=N_{(g,H_i)}(r)+o(T_f(r))$, 

$\mathrm{(ii)}$ \ $||\  N(r,|\nu^0_{(f,H_i)}-\nu^0_{(g,H_i)}|)+2N^{[1]}_{(h,H_i)}(r)=\sum_{t=1}^{2n+2}N^{[1]}_{(h,H_t)}(r)+o(T_f(r)),\ h\in\{f,g\},$ 

$\mathrm{(iii)}$ \ $||\ N(r,\min\{\nu^0_{(f,H_i)},\nu^0_{(g,H_i)}\})= \sum_{u=f,g} N^{[n]}_{(u,H_v)}(r)-nN^{[1]}_{(f,H_v)}(r)+o(T_f(r)).$

$\mathrm{(iv)}$ \ Moreover, if there exists an index $i_0\in\{1,...,2n+2\}$ such that $f$ and $g$ share $H_{i}$ with multiplicity truncated by level $n+1$ then 
$$\nu_{(f,H_{i_0})}(z)=\nu_{(g,H_{i_0})}(z)=n $$
 for all $z\in f^{-1}(H_{i_0})$ outside an analytic subset of counting function regardless of multiplicity is equal to $T_f(r)$.
\end{proposition}

\begin{proof}
(i)-(iii). For each two indecies $i$ and $j$, $1\le i<j\le 2n+2$, we defined
$$
  P_{ij}\overset{Def}{:=}\dfrac{(f,H_i)}{(g,H_i)}\cdot\dfrac{(g,H_j)}{(f,H_j)}.
$$
 By the supposition that the map $f\times g$ is algebraically nondegenerate,  we have that $P_{ij}$ is not constant.
Then by Proposition \ref{2.3} we have 
\begin{align*}
T(r,P_{ij})&=m(r,P_{ij})+N(r,\nu^{\infty}_{P_{ij}})=N(r,\nu^{\infty}_{P_{ij}})+o(T_f(r))\\ 
&= N(r,\nu^{\infty}_{\frac{(f,H_i)}{(g,H_i)}})+N(r,\nu^{\infty}_{\frac{(g,H_j)}{(f,H_j)}})+o(T_f(r))
\end{align*}
On the other hand, since $f=g$ and then $P_{ij}=1$ on $\bigcup_{\underset{t\ne i,j}{t=1}}^{2n+2}f^{-1}(H_t)$, therefore we have
$$
  N(r,\nu^{0}_{P_{ij}-1})\ge\sum_{\underset{t\ne i,j}{t=1}}^{2n+2}N^{[1]}_{(f,H_t)}(r).
$$
Since $N(r,\nu^{0}_{P_{ij}-1})\le T(r,P_{ij})$, we have
\begin{align}\label{2.8}
N(r,\nu^{\infty}_{\frac{(f,H_i)}{(g,H_i)}})+N(r,\nu^{\infty}_{\frac{(g,H_j)}{(f,H_j)}})\ge\sum_{\underset{t\ne i,j}{t=1}}^{2n+2}N^{[1]}_{(g,H_t)}(r)+o(T_f(r)).
\end{align}
Similarly, we also get
\begin{align}\label{2.9}
N(r,\nu^{\infty}_{\frac{(g,H_i)}{(f,H_i)}})+N(r,\nu^{\infty}_{\frac{(f,H_j)}{(g,H_j)}})\ge\sum_{\underset{t\ne i,j}{t=1}}^{2n+2}N^{[1]}_{(f,H_t)}(r)+o(T_f(r)).
\end{align}
It is also easy to see that
\begin{align}\label{2.10}
\begin{split}
N(r,\nu^{\infty}_{\frac{(f,H_t)}{(g,H_t)}})&+N(r,\nu^{\infty}_{\frac{(g,H_t)}{(f,H_t)}})=N(r,|\nu^0_{(f,H_t)}-\nu^0_{(g,H_t)}|)\\
&=N(r,\max\{\nu^0_{(f,H_t)},\nu^0_{(g,H_t)}\})-N(r,\min\{\nu^0_{(f,H_t)},\nu^0_{(g,H_t)}\})\\
&=N(r,\max\{\nu^0_{(f,H_t)},\nu^0_{(g,H_t)}\})+N(r,\min\{\nu^0_{(f,H_t)},\nu^0_{(g,H_t)}\})\\
&\ \ -2N(r,\min\{\nu^0_{(f,H_t)},\nu^0_{(g,H_t)}\})\\
&=N_{(f,H_t)}(r)+N_{(g,H_t)}(r)-2N(r,\min\{\nu^0_{(f,H_t)},\nu^0_{(g,H_t)}\}), \forall 1\le t\le 2n+2.
\end{split}
\end{align}
Therefore, by summing-up both sides of  (\ref{2.7}) and (\ref{2.8}) we get
\begin{align}\label{2.11}
\sum_{v=i,j}\bigl (\sum_{u=f,g}N_{(u,H_v)}(r)-2N(r,\min\{\nu^0_{(f,H_v)},\nu^0_{(g,H_v)}\})\bigl )\ge\sum_{u=f,g}\sum_{\underset{t\ne i,j}{t=1}}^{2n+2}N^{[1]}_{(u,H_t)}(r)+o(T_f(r)).
\end{align}
Since 
\begin{align}\label{2.12}
 T_u(r)\ge N_{(u,H_t)}(r),\ u=f,g,
\end{align}
the above inequality yields that
\begin{align}\label{2.13}
\begin{split}
||\ \sum_{u=f,g}2T_u(r)&\ge\sum_{u=f,g}\bigl (N_{(u,H_i)}(r)+N_{(u,H_j)}(r)\bigl )\\
&\ge \sum_{v=i,j}2N(r,\min\{\nu^0_{(f,H_v)},\nu^0_{(g,H_v)}\})+\sum_{u=f,g}\sum_{\underset{t\ne i,j}{t=1}}^{2n+2}N^{[1]}_{(u,H_t)}(r)+o(T_f(r)).
\end{split}
\end{align}
We also see that for all $z\in f^{-1}(H_t), v=i,j,$ 
\begin{align*}
\min\{\nu^0_{(f,H_v)}(z),\nu^0_{(g,H_v)}(z)\}\geqslant \min\{\nu^0_{(f,H_v)},n\}+\min\{\nu^0_{(g,H_v)}(z),n\}-n.
\end{align*}
This implies that
\begin{align}\label{2.14}
\begin{split}
N(r,\min\{\nu^0_{(f,H_v)},\nu^0_{(g,H_v)}\})&\ge \sum_{u=f,g} N^{[n]}_{(u,H_v)}(r)-nN^{[1]}_{(f,H_v)}(r)\\
&=\sum_{u=f,g} \bigl (N^{[n]}_{(u,H_v)}(r)-\dfrac{n}{2}N^{[1]}_{(u,H_v)}(r)\bigl ).
\end{split}
\end{align}
Combining inequalities (\ref{2.13}) and (\ref{2.14}), we have
\begin{align*}
||\ \sum_{u=f,g}2T_u(r)&\ge\sum_{v=i,j}\sum_{u=f,g} \bigl (2N^{[n]}_{(u,H_v)}(r)-N^{[1]}_{(u,H_v)}(r)\bigl )\\ 
& \ \ +\sum_{u=f,g}\sum_{\underset{t\ne i,j}{t=1}}^{2n+2}N^{[1]}_{(u,H_t)}(r)+o(T_f(r)).
\end{align*}
Summing-up both sides of the above inequality over all pair $(i,j), i\ne j$, and using the Second Main Theorem, we get
\begin{align*}
||\ \sum_{u=f,g}2T_u(r)&\ge\dfrac{2}{n+1}\sum_{u=f,g}\sum_{v=1}^{2n+2} \bigl (2N^{[n]}_{(u,H_v)}(r)+o(T_f(r))\\ 
& \sum_{u=f,g}\ge\dfrac{2(2n+2-n-1)}{n+1}T_u(r)+o(T_f(r))=\sum_{u=f,g}2T_u(r)+o(T_f(r)).
\end{align*}
The last equality yields that all inequalities (\ref{2.8}) (\ref{2.9}) and (\ref{2.12}-\ref{2.12}) become equalities outside a Borel set of finite measure. Summarizing all these ``\textit{equalities}", we obtain the following:
\begin{align}\label{2.15}
&||\ T_f(r)=N_{(f,H_i)}(r)+o(T_f(r)) \text{ and }||\ T_g(r)=N_{(g,H_i)}(r)+o(T_f(r))\  \text{(by (\ref{2.12}))},\\ 
\label{2.16}
&||\ \sum_{v=i,j}N(r,|\nu^0_{(f,H_v)}-\nu^0_{(g,H_v)}|)=\sum_{u=f,g}\sum_{\underset{t\ne i,j}{t=1}}^{2n+2}N^{[1]}_{(u,H_t)}(r)+o(T_f(r)) \text{(by (\ref{2.10}) and (\ref{2.11}))},\\ 
\label{2.17}
&||\ N(r,\min\{\nu^0_{(f,H_i)},\nu^0_{(g,H_i)}\})= \sum_{u=f,g} N^{[n]}_{(u,H_v)}(r)-nN^{[1]}_{(f,H_v)}(r)\ \text{(by (\ref{2.10}) and (\ref{2.11}))},
\end{align}
for every $i=1,...,2n+2.$ Then, equalities (\ref{2.15}) and (\ref{2.16}) prove the first assertion and the third assertion of the proposition.
Also the equality (\ref{2.16}) implies that 
$$ 
 ||\ \sum_{v=i,j}\bigl (N(r,|\nu^0_{(f,H_v)}-\nu^0_{(g,H_v)}|)+2N^{[1]}_{(h,H_v)}(r)\bigl )=\sum_{u=f,g}\sum_{t=1}^{2n+2}N^{[1]}_{(u,H_t)}(r)+o(T_f(r)) 
$$
holds for all $i,j\in\{1,...,2n+2\}$ and $h\in\{f,g\}$, it easily follows that
$$ 
||\  N(r,|\nu^0_{(f,H_i)}-\nu^0_{(g,H_i)}|)+2N^{[1]}_{(h,H_i)}(r)=\sum_{t=1}^{2n+2}N^{[1]}_{(h,H_t)}(r)+o(T_f(r)), 1\le i\le 2n+2,\ h\in\{f,g\}. 
$$
Then the second assertion is proved.

(iv). Without loss of generality, we may assume that $i_0=2n+2$. From the third assertion and the assumption that $f$ and $g$ share $H_{2n+2}$ with multiplicity truncated by level $n+1$, we have
\begin{align*}
||\ N^{[n+1]}_{(f,H_{2n+2})}(r)&\le N(r,\min\{\nu_{(f,H_{2n+2})},\nu_{(g,H_{2n+2})}\})\\
&=\sum_{u=f,g}N^{[n]}_{(u,H_{2n+2})}(r)-nN^{[1]}_{(g,H_{2n+2})}(r)+o(T_f(r))\\ 
&\le\sum_{u=f,g}N^{[n]}_{(u,H_{2n+2})}(r)-N^{[n]}_{(g,H_{2n+2})}(r)+o(T_f(r))\\
&=N^{[n]}_{(f,H_{2n+2})}(r)+o(T_f(r)).
\end{align*}
This yields that 
$$
  ||\ N^{[n+1]}_{(f,H_{2n+2})}(r)=N^{[n]}_{(f,H_{2n+2})}(r)+o(T_f(r))\text{ and }||\ N^{[n]}_{(u,H_{2n+2})}(r)=nN^{[1]}_{(g,H_{2n+2})}(r)+o(T_f(r)).
$$
It folows that 
$$ 
  \min\{\nu_{(f,H_{2n+2})},n+1\}=  \min\{\nu_{(f,H_{2n+2})},n\}\text{ and } \min\{\nu_{(g,H_{2n+2})},n\}= n\min\{\nu_{(f,H_{2n+2})},n1\}
$$
outside an analytic subset $S$ of counting function regardless of multiplicity is equal to $T_f(r)$. Therefore, 
$$ \nu_{(f,H_{2n+2})}(z)\le n\text{ and } \nu_{(g,H_{2n+2})}(z)\ge n\  \forall z\in f^{-1}(H_{2n+2})\setminus S.$$
Similarly, we have
$$ \nu_{(g,H_{2n+2})}(z)\le n\text{ and } \nu_{(f,H_{2n+2})}(z)\ge n $$
for all $z\in f^{-1}(H_{2n+2})$ outside an analytic subset $S'$ of counting function regardless of multiplicity is equal to $T_f(r)$. Then we have
$$ \nu_{(f,H_{2n+2})}(z)=\nu_{(g,H_{2n+2})}(z)=n $$
for all $z\in f^{-1}(H_{2n+2})\setminus (S\cup S')$. The fourth assertion is proved.
\end{proof}

\textbf{Proof of  Main Theorem.}\ Suppose that $f\times g$ is not algebraically degenerate. 
 Then by Lemma \ref{2.7}(ii)-(iv) and by the assumption, we have the following:
\begin{align*}
||\ 2 N^{[1]}_{(h,H_{2n+2})}(r)=\sum_{t=1}^{2n+2}N^{[1]}_{(h,H_{t})}(r)+o(T_f(r))\ h\in\{f,g\}.
\end{align*}
By the Second Main Theorem, it follows that
\begin{align*}
||\ T_h(r)\ge &N^{[1]}_{(h,H_{2n+2})}(r)=\sum_{\overset{t=1}{t\ne 2n+2}}^{2n+2}N^{[1]}_{(h,H_{t})}(r)+o(T_f(r))\\ 
& \ge\dfrac{1}{n}\sum_{\overset{t=1}{t\ne 2n+2}}^{2n+2}N^{[n]}_{(h,H_{t})}(r)+o(T_f(r))\\ 
&\ge\dfrac{2n+1-n-1}{n}T_h(r)+o(T_f(r))=T_h(r)+o(T_f(r))\\
\end{align*}
for each $h\in\{f,g\}$. Therefore, we easily obtain that
\begin{align*}
||\ T_h(r) &= N_{(h,H_{2n+2})}(r)+o(T_h(r))=N^{[n]}_{(h,H_{2n+2})}(r)+o(T_h(r))\\
&=N^{[1]}_{(h,H_{2n+2})}(r)+o(T_h(r)),\ \forall h\in\{f,g\}.
\end{align*}
Then, by Lemma \ref{2.7}(iii), we have
\begin{align*}
||\ T_h(r)&=N(r,\min\{\nu^0_{(f,H_i)},\nu^0_{(g,H_i)}\})= \sum_{u=f,g} N^{[n]}_{(u,H_v)}(r)-nN^{[1]}_{(h,H_v)}(r)+o(T_h(r))\\
&=2T_h(r)-nT_h(r)+o(T_h(r)),\ \forall h\in\{f,g\}.
\end{align*}
Letting $r\longrightarrow +\infty$, we get $n=1$. This is a contradiction to the assumption that $n\ge 2.$ Therefore, the supposition is impossible. Then the map $f\times g$ is algebraically degenerate.\hfill$\square$

\section{Two meromorphic mappings with two family of hyperplanes}
Let $f$ and $g$ be three distinct meromorphic mappings of $\C^m$ into $\P^n(\C)$. Let $\left\{H_i\right\}_{i=1}^{2n+2}$ and $\left\{G_i\right\}_{i=1}^{2n+2}$  be two families of hyperplanes of $\P^n(\C)$ in general position. Hyperplanes $H_i$ and $G_i$ are given by
\begin{align*}
H_i&=\{(\omega_0:\cdots :\omega_n)\ |\ \sum_{v=0}^na_{iv}\omega_v=0\}\\
\text{and }G_i&=\{(\omega_0:\cdots :\omega_n)\ |\ \sum_{v=0}^nb_{iv}\omega_v=0\}
\end{align*} 
respectively. Let $f=(f_0:\cdots :f_n)$ and $g=(g_0:\cdots :g_n)$ be reduced representations of $f$ and $g$ respectively. We set
$$ 
(f,H_i)=\sum_{v=0}^na_{iv}f_v\text{ and }(g,G_i)=\sum_{v=0}^nb_{iv}g_v.
$$ 
In this section, we will consider the case of two meromorphic mappings sharing two different families of hyperplanes as follows.

\begin{theorem}
Let $f,g,\left\{H_i\right\}_{i=1}^{2n+2}$ and $\left\{G_i\right\}_{i=1}^{2n+2}$ be as above. Assume that 

$\mathrm(a)$ $\dim f^{-1}(H_i)\cap f^{-1}(H_j)\leqslant m-2 \ \forall 1\leqslant i<j\leqslant 2n+2,$

$\mathrm(b)$ $f^{-1}(H_i)=g^{-1}(G_i),$ for $k=1,2,$ and $i=1,...,2n+1,$

$\mathrm(c)$ $\min\{\nu_{(f,H_{2n+2})},n+1\}=\min\{\nu_{(g,G_{2n+2})},n+1\}$,

$\mathrm(d)$ $\dfrac{(f,H_v)}{{(f,H_j)}}=\dfrac{(g,G_v)}{{(g,G_j)}}$ on $\bigcup_{i=1}^{2n+2}f^{-1}(H^0_i)\setminus f^{-1}(H^0_j),$ for $1\leqslant v,j\leqslant 2n+2.$

If $n\ge 2$ then the map $f\times g$ is algebraically degenerate.
\end{theorem}
\begin{proof}
We consider the linearly projective transformation $\mathcal L$ of $\P^n(\C)$ which is given by $\mathcal L((z_0:\cdots :z_n))=(\omega_0:\cdots :\omega_n)$ with
\begin{align*}
\left (
\begin{array}{ccc}
\omega_0\\ 
\vdots\\
\omega_n
\end{array}\right)
=\underbrace{\left (\begin{array}{ccc}
c_{10}&\cdots&c_{1n}\\ 
\vdots&\cdots&\vdots\\
c_{(n+1)0}&\cdots&c_{(n+1)n}
\end{array}\right)}_{C}
\left (
\begin{array}{ccc}
z_0\\ 
\vdots\\
z_n
\end{array}\right),
\end{align*}
where
$$ 
C={\underbrace{\left(
\begin{array}{ccc}
a_{10}&\cdots&a_{1n}\\ 
\vdots&\cdots&\vdots\\
a_{(n+1)0}&\cdots&a_{(n+1)n}
\end{array}
\right)}_{A}}^{-1}\cdot
\underbrace{\left (
\begin{array}{ccc}
b_{10}&\cdots&b_{1n}\\ 
\vdots&\cdots&\vdots\\
b_{(n+1)0}&\cdots&b_{(n+1)n}
\end{array}
\right)}_{B}
$$

We set 
$$
(a'_{i0},...,a'_{in})=(b_{i0},...,b_{in})\cdot C^{-1},\ \text{ for }i=1,..,2n+2.
$$
 Since $A\circ B=C,$ then
$$
(a'_{i0},...,a'_{in})=(a_{i0},...,a_{in}),\ \forall i=1,...,n+1.
$$ 
Suppose that there exists an index $i_0\in\{n+2,...,2n+2\}$ such that 
$$ 
(a'_{i0},...,a'_{in})\ne (a_{i0},...,a_{in}).
$$ 
We consider the following function
$$ 
F=\sum_{j=0}^{n}(a'_{i_0j}-a_{i_0j})f_j.
$$
Since $f$ is linearly nondegenerate, $F$ is a nonzero meromorphic function on $\C^m$. For $z\in\bigcup_{i=1}^{2n+2}f^{-1}(H_i)\setminus I(f^0)$,  without loss of generality we may assume that $(f,H_1)(z)\ne 0$, then 
\begin{align*}
F(z)=&\sum_{j=0}^{n}(a'_{i_0j}-a_{i_0j})f_j(z)=(a_{i_00},...,a_{i_0n})\cdot C^{-1}(f)(z)-(f,H_{i_0})(z)\\
=&(a_{i_00},...,a_{i_0n})\cdot B^{-1}\circ A(f)(z)-(f,H_{i_0})(z)\\
=&\dfrac{(a_{i_00},...,a_{i_0n})\cdot B^{-1}\circ A(f)(z)-(f,H_{i_0})(z)}{(f,H_1)(z)}\cdot (f,H_1)(z)\\
=&\dfrac{(a_{i_00},...,a_{i_0n})\cdot B^{-1}\circ B(f)(z)-(f,H_{i_0})(z)}{(f,H_1)(z)}\cdot (f,H_1)(z)\\
=&\dfrac{(a_{i_00},...,a_{i_0n})(f)(z)-(f,H_{i_0})(z)}{(f,H_1)(z)}\cdot (f,H_1)(z)\\
=&\dfrac{(f,H_{i_0})(z)-(f,H_{i_0})(z)}{(f,H_{1})(z)}\cdot (f,H_1)(z)=0.
\end{align*}
Therefore, it follows that
$$ 
N_{F}(r)\geqslant\sum_{i=1}^{2n+2}N_{(f,H_i)}^{[1]}(r).
$$
On the other hand, by Jensen formula we have that
$$ 
N_{F}(r)=\int\limits_{S(r)}\log |F(z)|\ \eta+O(1)\leqslant \int\limits_{S(r)}\log ||f(z)||\ \eta+O(1)= T_{f}(r)+o(T_{f}(r)).
$$
By using the Second Main Theorem, we obtain
\begin{align*} 
||\ (n+1)T_{f}(r)&\leqslant \sum_{i=1}^{2n+2}N_{(f,H_i)}^{[n]}(r)+o(T_{f}(r))\\
&\leqslant n\sum_{i=1}^{2n+2}N_{(f,H_i)}^{[1]}(r)+o(T_{f}(r))\leqslant nT_{f}(r).
\end{align*}
It implies that $||\ T_{f}(r)=0.$ This is a contradiction to the fact that $f$ is linearly nondegenerate. Therefore we have
$$
(a'_{i0},...,a'_{in})=(a_{i0},...,a_{in}),\ \forall i=1,...,2n+2.
$$ 
Hence $\mathcal L(G_i)=H_i$ for all $i=1,...,2n+2.$

We set $\tilde g=\mathcal L (g),\ k=1,2$. Then $f$ and $\tilde g$ share $\{H_1,..,H_{2n+1}\}$ ignoring multiplicity and share $H_{2n+2}$ with multiplicity truncated by level $n+1$. By Main Theorem, the map $f\times\tilde g$ is algebraically degenerate. We easily see that the map
$$
\begin{array}{cccc}
  \Psi: &\P^n(\C)\times \P^n(\C)&\longrightarrow&\P^n(\C)\times \P^n(\C)\\ 
&((\omega_0:\cdots :\omega_n)\times (z_0:\cdots :z_n))&\mapsto&((\omega_0:\cdots :\omega_n)\times \mathcal L^{-1}(z_0:\cdots :z_n))
\end{array}
$$
is an automorphism of $\P^n(\C)\times \P^n(\C)$. Therefore, the map $f\times g=\Phi (f\times\tilde g)$ is an algebraically degenerate mapping into $\P^n(\C)\times \P^n(\C)$. The theorem is proved.
\end{proof}

\end{document}